\RequirePackage{fix-cm}
\documentclass{svjour3}                     % onecolumn (standard format)
\smartqed  % flush right qed marks, e.g. at end of proof
\usepackage{graphicx}
 \usepackage{mathptmx}      % use Times fonts if available on your TeX system
\usepackage{amsmath}
\usepackage{amsfonts}
\usepackage{amssymb}
\usepackage{xcolor}

\newcommand{\R}{\mathbb{R}}

\journalname{Bulletin of Mathematical Biology}
\begin{document}

\title{Stabilization of structured populations via vector target oriented control
\thanks{E. Braverman was partially supported by the NSERC research grant RGPIN-2015-05976.
D.~Franco  was partially supported by the Spanish Ministerio de Econom\'{\i}a y Competitividad and FEDER, grant 
MTM2013-43404-P.}
%\thanks{Grants or other notes
%about the article that should go on the front page should be
%placed here. General acknowledgments should be placed at the end of the article.}
}
%\subtitle{Do you have a subtitle?\\ If so, write it here}

\titlerunning{Stabilization via vector target oriented control}        % if too long for running head

\author{Elena Braverman       \and
         Daniel Franco  %etc.
}

%\authorrunning{Short form of author list} % if too long for running head

\institute{E. Braverman \at
              Department of Mathematics and Statistics, University of
              Calgary, 2500 University Drive N.W., Calgary,
               AB T2N 1N4, Canada \\     
              \email{maelena@ucalgary.ca}           %  \\
%             \emph{Present address:} of F. Author  %  if needed
           \and
           D. Franco \at
              Departamento de Matem\'atica Aplicada, E.T.S.I. 
              Industriales, Universidad Nacional de Educaci\'on a Distancia~(UNED), c/~Juan del Rosal~12, 28040, Madrid, Spain\\
              \email{dfranco@ind.uned.es}
}

\date{Received: date / Accepted: date}
% The correct dates will be entered by the editor

\maketitle

\begin{abstract}
In contrast with unstructured 
models,  structured discrete 
population models have been able to fit and predict chaotic experimental data. 
However, most of the chaos control techniques in the literature have been designed and analyzed in a one-dimensional setting.  
Here, by introducing target oriented control for discrete dynamical systems,
we prove the possibility to stabilize a chosen state for a wide range of structured population models. 
The results are illustrated with introducing a control in the celebrated LPA model describing a flour beetle dynamics. 
Moreover, we show that the new control allows to stabilize periodic solutions for higher order difference equations, 
such as the delayed Ricker model, for which previous target oriented methods were not designed.

\keywords{Target oriented control \and Discrete population models \and 
Structured populations \and LPA model \and Delay Ricker model}
% \PACS{PACS code1 \and PACS code2 \and more}
% \subclass{MSC code1 \and MSC code2 \and more}
\end{abstract}

\section{Introduction}

In 1970ies, R. May showed in \cite{may1974biological,May1976} that certain one-dimensional discrete 
population models commonly used in theoretical ecology, can exhibit chaos  for some parameter values.
This fact immediately started two extremely interesting lines of research. 
The first line, and probably the hardest, is related to finding chaotic models fitting experimental data and, 
which is more important, predicting experiment outcomes. 
May himself, along with Hassel and Lawton, observed that there are species for which, in order to fit experimental 
data to a simple one-dimensional model, we have to consider parameters that generate chaotic dynamics 
\cite{hassell1976patterns}.
The task of predicting experimental data using these simple one-dimensional models appears to be hard \cite{morris1990problems} and, 
as far as we know, is an open problem. 
However, more elaborated  models, e.g. including  age structure or the interaction between different species, 
were able to fit and predict experimentally the detected chaos \cite{Costantino389,kot1993population}. 
The second line of research deals with whether and how this complexity can be controlled  \cite{sole1999controlling}. 
Several strategies have been proposed for controlling chaos in population dynamics, e.g. 
\cite{braverman2012stabilization,capeans2014less,Dattani,desharnais2001chaos,hilker2005control,liz2010control,parthasarathy1995controlling,sah2013stabilizing,segura2016adaptive,tung2014comparison}. 
Most of the control techniques proposed for discrete dynamical systems have been introduced and studied for scalar maps. 
However, the outcome of the study of experimental data and model-data matching process 
points out that stabilizing multidimensional maps is more of an interest and practical application. 

In the present paper, our purpose is to stabilize nonlinear chaotic dynamical systems given by the first order vector 
difference equation
\begin{equation}
\label{1}
\mathbf{x}_{n+1}=\mathbf{f} (\mathbf{x}_n), \quad n= 0,1,2,\dots , 
\end{equation}     
where $\mathbf{f}\colon D \to D$ is a continuous function, $D$ is a convex subset 
of $\R^d$, and $\mathbf{x}_0 \in D$ is the initial condition. 
This type of systems is suitable to describe multi-species models, 
as well as single species populations with a structure (e.g., an age structure given by different age stages, such as 
juvenile and adults, or a spatial structure given by a population living in different patches connected by dispersal), as 
well as physical models.
System \eqref{1} can exhibit chaotic behavior \cite{Marotto}, and chaos control of multidimensional systems 
is usually a harder problem than for scalar maps. 
See, for example,  the recent paper  \cite{LizPotzsche} on multidimensional prediction based control, where 
the challenges to stabilize the two-dimensional H\'{e}non map \cite{Henon1976} were outlined.

Here, we consider a natural extension of a method called  \emph{target oriented control} (TOC). 
This method was introduced in \cite{Dattani} for first order difference equations, i.e. when $d=1$. 
TOC establishes a target state and implements an increase or a decrease of the state variable each time step, 
depending on  whether its value exceeds or is below the target state 
\begin{equation}
\label{TOC_eq}
x_{n+1}=f( c T+(1-c)x_n ), \quad T\ge 0, \; c\in [0,1).
\end{equation}

In a population, to apply TOC we fix a target population size $T$ and cull/restock a fraction $c$ 
of the difference between the target 
and the actual population sizes. 
This fraction $c$ measures the control intensity. 
Therefore, TOC is a two-parameter control method in which the controller chooses $T$ and $c$. 
It is very interesting that recent experiments with fruit flies support that TOC, 
as other two-parameter control methods, has 
a better performance than one-parameter control techniques in enhancing simultaneously 
different ecological stability properties  \cite{tung2016simultaneous}. 
Indeed, in such experiments TOC concurrently reduced population fluctuations, 
decreased extinction probability and increased effective population size.

Some of these stabilization properties were explained  mathematically in \cite{Chaos2014,TPC}.  
For instance, it was proved that if the control intensity $c$ is close enough to one,  
TOC can provide global stabilization of a positive equilibrium for a wide class of smooth maps. 
In \cite{TPC}, it was also noticed  that if we describe the linear transformation of the variable
\begin{equation}
\label{phidef}
\phi(x)=cT+(1-c)x
\end{equation}
and consider the modified target oriented control (MTOC) 
\begin{equation}
\label{MTOC_eq}
x_{n+1}=cT+(1-c)f(x_n), \quad T\ge 0, \; c\in [0,1),
\end{equation}
then the global (local) asymptotic stability of the equilibrium $K_c$ of \eqref{TOC_eq} is 
equivalent to the global (local) asymptotic stability of the equilibrium $P_c=\phi(K_c)$ 
of \eqref{MTOC_eq}. 
In other words,  
the stability of the equilibrium is not altered by switching the time of control application: before or after
the reproduction period.

Here, we consider the natural extensions of TOC and MTOC to multidimensional systems: Vector  Target Oriented Control (VTOC)
\begin{equation}
\label{2a}
\mathbf{x}_{n+1}=\mathbf{f}(c\mathbf{T}+(1-c)\mathbf{x}_n), \quad \mathbf{T}\in D, \; c\in [0,1),
\end{equation}
as well as Vector Modified Target Oriented Control (VMTOC)
\begin{equation}
\label{2}
\mathbf{x}_{n+1}=c\mathbf{T}+(1-c)\mathbf{f}(\mathbf{x}_n), \quad \mathbf{T} \in D, \; c\in [0,1).
\end{equation}
Our main results give sufficient conditions for the local and global stabilization of either an equilibrium
or some other target vector state 
and estimate the minimum control intensity $c^{\ast}$ to attain such stabilization using these new methods. 
Regular \eqref{2a} and modified \eqref{2} vector target oriented controls are topologically conjugate 
(see Lemma \ref{lemma3} in the Appendix), which implies that from a stability perspective both systems have the same properties. 
Since the results are focused on  stability properties, from now on we restrict ourselves to VMTOC without loss of 
generality.

The paper \cite{cushing2002chaos} is perhaps the best example of the intersection of the two lines of research described in 
the first paragraph of this introduction. 
There, Cushing et al. put forward a control method to stabilize fluctuations of an insect population with three 
age stages: larvae, pupae and adults. 
The method was based on the analysis of the chaotic attractor of the theoretical model 
known to describe the population dynamics. 
In order to illustrate our results, we consider the same model, called LPA model, 
showing that, at least theoretically, VMTOC globally stabilizes an equilibrium if the control intensity is high enough. 
The main advantages of VMTOC, compared to the method presented in \cite{cushing2002chaos}, are its simplicity, since VMTOC does 
not need any information about the chaotic attractor, and flexibility, since any age-stage configuration can be stabilized. 

We also use the LPA model to illustrate that the selection of the target can have important consequences, 
not previously reported for target oriented control. 
We show that depending on the selection of the target vector $\mathbf{T}$, 
an increase of the control intensity $c$ may not always be stabilizing due to the presence of bubbles. 

VMTOC is not the first extension of target oriented control.  
Indeed, TOC has been recently extended to higher order difference equations \cite{braverman2015stabilization}. 
Equations of this type naturally arise when considering populations with non-overlapping generations but with 
multi-seasonal interactions \cite{levin1976note}. 
Since higher order difference equations can be rewritten as a vector map, 
our results here are also applicable to this particular case. 
If the target $\mathbf{T}$ is a vector with equal coordinates, the 
stabilization scheme of \cite{braverman2015stabilization} can be obtained as a particular case of the results of the present 
paper.
However, stabilizing  a certain vector state in $\R^d$ with non-equal coordinates 
corresponds to a $d$-periodic orbit stabilization for the original higher order difference equation. 
This opens the possibility to stabilize periodic orbits in delayed population models, 
which seemed not possible with the approach presented in \cite{braverman2015stabilization}. 

Let us note that the method developed in the present paper allows to effectively deal not only with chaotic 
but also with multi-stable systems, as well as control oscillation amplitudes.

The paper is organized as follows. In Section~\ref{Sec_calculus} some auxiliary results are collected: the fact that we can 
combine two or more VMTOCs to obtain a control of the same type,
and that, with an appropriate combination of the target vector $\mathbf{T}$ and the control intensity $c$, 
we can get any vector in the  interior of the domain of $\mathbf{f}$ as an equilibrium (in 
fact, there is an infinite number of such $c,\mathbf{T}$ pairs). 
Moreover, we present sufficient conditions for VMTOC to have at least one non-trivial equilibrium for all control intensities 
$c\in (0,1)$.  
In Section \ref{Sec_stability} we justify the possibility of local and global stabilization of an equilibrium of the uncontrolled system 
or of any prescribed vector in $D$. 
Section \ref{Applic} illustrates the results with two different applications: LPA model and the delayed Ricker model.
The discussion section summarizes the results obtained and further possible developments. 
The auxiliary result on the equivalence of VTOC and VMTOC is postponed to the appendix.

\section{Calculus of VMTOCs}
\label{Sec_calculus}

Since we are interested in the capacity of VMTOC to stabilize an equilibrium, 
we begin by showing that under quite general conditions on 
$\mathbf f$ and $D$, such an equilibrium exists. 
Indeed, if $\mathbf{f}: D \to D$ is continuous and the set $D$ is convex and compact, 
then Brouwer Fixed-point Theorem %(e.g. \cite[Proposition 2.6]{Zeidler1986Fixed}) 
implies the existence of at least one equilibrium point of VMTOC in $D$  for every $c \in (0, 1)$  and any $\mathbf{T} \in D$. 
However,  $D$ is not always compact. 
For example, in population models, where each component of $\mathbf{x}$ corresponds to a certain population size, 
it is natural to assume  that $D={\mathbb R}^d_{+}$, ${\mathbb R}_{+}=[0,\infty)$ but also that, due to the competition for 
resources, the inequality 
$\|\mathbf{f(x)}\|\le \|\mathbf{x}\|$ holds when $\|\mathbf{x}\|$ is large, where  $\|\cdot\|$ is an arbitrary fixed norm in $\R^d$.  
The next result shows that in this situation VMTOC has a nontrivial  equilibrium for every $c \in (0, 1)$ as well.

\begin{lemma} 
	\label{lemma_exist}
	Assume that $\mathbf{f}\colon \mathbb{ R}^d_+\to \mathbb{ R}^d_+$ is continuous and there exists a positive constant $M$ such that 
	$\mathbf{\|f(x)\| \le \| x\|}$ holds for any $\mathbf{x}\in \mathbb{ R}^d_+$ with $\|\mathbf{x}\|\ge M$. 
	Then for any  target $\mathbf{T}\in \mathbb{R}^d_{+}\setminus \{\mathbf{0} \}$ and $c \in (0, 1)$, 
	there exists at least one equilibrium $\mathbf{x}^{\ast} \in \mathbb R^d_+$ of VMTOC satisfying
	$
	0<\|{\mathbf x}^{\ast}\|<\max\{M,\|\mathbf{T}\|\}. 
	$
	Moreover, if all the components of $\mathbf{T}$ are non-zero, then for all $c \in (0, 1)$ the equilibria of  VMTOC have all their components positive. 
\end{lemma}

\begin{remark}
\label{remark1}
It is interesting to note the practical consequences of Lemma~\ref{lemma_exist}.
A controller can select the target $\mathbf{T}$ depending on the aimed result. 
If the system \eqref{1} models the interactions among different species, 
and we aim to eradicate one of them while keeping the others, 
then it is natural to select $\mathbf{T}$ with a zero component, 
and such that the fixed point belongs to the boundary of $\mathbb R^d_+$. 
Whereas if the system \eqref{1} models the interactions between different age stages of the same species, 
and we aim to reduce the fluctuations,  
then $\mathbf T$ will belong to the interior of  $\mathbb R^d_+$ and, by Lemma \ref{lemma_exist},  the equilibria of VMTOC as well.
\end{remark}

\begin{remark}
\label{remark_add}
Note that the original map $\mathbf{f}$ in Lemma~\ref{lemma_exist} does not need to have a fixed point  
in $\mathbb{R}^d_{+}\setminus \{\mathbf{0} \}$, for example, 
the unique fixed point of $\mathbf{f}(\mathbf{x})=\alpha \mathbf{x}$, $\alpha \in (0,1)$, is $\mathbf{0}$.
However, with any nontrivial target $\mathbf T$, the equilibrium of the controlled map becomes non-trivial.
\end{remark}

Next, we consider the possibility to have an arbitrary vector state in  the interior of 
$D$ as an equilibrium in controlled model \eqref{2}.
\begin{lemma}
	\label{lemma2}
	Let $\mathbf{f} \colon D \to D$ be a continuous function, where $D$ is convex. Then for any $\mathbf{K} \in D$, $\mathbf{K} \not\in \partial D$ there exist 
	$c_\mathbf{K}\in (0,1)$
	 and $\mathbf{T_K}\in D$ such that $\mathbf{K}$ is an equilibrium of $\mathbf{g}(\mathbf{x})=c_\mathbf{K} \mathbf{T_K}+(1-c_\mathbf{K})\mathbf{f}(\mathbf{x})$.
\end{lemma}

As easily follows from the proof of Lemma~\ref{lemma2}, there are infinitely many pairs $(c_\mathbf{K},\mathbf{T_K})$. For any $c_\mathbf{K} \in (0,1)$,
$\mathbf{T_K}$ is unique, but depending on the geometry of $D$, $\mathbf{T_K}$ can be an arbitrary point on the ray starting at $\mathbf{K}$ and $\mathbf{K}-\mathbf{f(K)}$-directed. 
The closer is $c_\mathbf{K}$ to 1, the smaller the distance between $\mathbf{T_K}$ and $\mathbf{K}$ is. 
In other words, we could choose either a 
larger target and a weaker control, or a smaller target but a tighter control. 
Of course, this election can affect the stability of the equilibrium $\mathbf{K}$ of $\mathbf{g}$. 
In the following section we will consider such stability.  
But before that, we present our last auxiliary result on the effect of combining two different VMTOCs.

\begin{lemma}
	\label{lemma1}
	Let $D$ be convex. Then a combination of two VMTOCs is a VMTOC.
\end{lemma}

% % % % % % % % % % % % %
% % % % % % % % % % % % %

\section{Stabilization of an equilibrium}
\label{Sec_stability}

\subsection{Local stabilization }

When a sufficiently strong control is implemented,   
VMTOC  locally  asymptotically stabilizes any finite equilibrium.

\begin{theorem}
	\label{p_esta}
	Assume that $\mathbf{f}=(f_1,\dots,f_d)\colon D \to D$ is continuously differentiable and that, for a fixed compact set $S\subset D$,
	VMTOC with a target $\mathbf{T}\in D$ has at least one equilibrium in $S$ for every $c \in (0, 1)$. Define $A=\min \{ B,C \} $,
	where
	\[
	B=  \max_{\mathbf{x}\in S,\ j=1,\dots, d} \sum_{i=1}^{d} \left | \frac{\partial f_{j}}{\partial x_i}(\mathbf{x})\right |
	\quad 
	\mbox{ and } 
	\quad
	C=\max_{\mathbf{x}\in S,\ j=1,\dots,d}\sum_{i=1}^{d} \left | \frac{\partial f_{i}}{\partial x_j}(\mathbf{x}) \right | 
	\]
	
	Then, all equilibria of VMTOC in $S$ are  locally asymptotically stable for $c\in  (c^*, 1)$ with 
	\begin{equation}
	\label{boundc}
	c^* = \max\left\{ 0,1-\frac{1}{A} \right\}.
	\end{equation}   
\end{theorem}

\begin{remark}
	\label{remark_loc_stability}
	If $\mathbf{T}$ is chosen as an equilibrium $\mathbf{K}$ of $\mathbf{f}$, then $\mathbf{K}$ is also an equilibrium of VMTOC 
	for every $c \in (0, 1)$. In such a case,  in Theorem \ref{p_esta} we can take $S=\{\mathbf{K}\}$,  and in the 
	definition of $c^*$ we can replace $A$ by $\rho(J\mathbf{f}(\mathbf{K}))$, 
	where $\rho$ denotes the spectral radius of a matrix. 
\end{remark}

\subsection{Global stabilization}

First, we consider the case when an equilibrium of the uncontrolled system is to be stabilized, that is, it satisfies 
\begin{equation}   
\label{fixed}
\mathbf{K=f(K)}.
\end{equation}
In this section, we assume that there exists $L>0$ such that
\begin{equation} 
\label{3}
\| \mathbf{f(x)-K} \| \leq L \| \mathbf{x-K} \|, ~~\mathbf{x} \in D.
\end{equation} 
Note that  \eqref{3} yields that $\mathbf{K}$ is a fixed point of $\mathbf{f}$.

Our first result  gives not only a sufficient global stabilization 
condition but also an estimate of the control 
intensity necessary to achieve it, which depends on the constant $L$ in condition \eqref{3}.  

\begin{theorem}
	\label{th1}
	Assume that for a continuous function $\mathbf{f} \colon D \to D$ equality \eqref{fixed} and  inequality \eqref{3} hold, where $D 
\subseteq \R^d$ is convex, $\mathbf{K} \in D$ and $\mathbf{T=K}$. 
	
	If $L<1$, then all solutions of \eqref{2} with an initial condition $\mathbf{x}_0 \in D$ converge to $\mathbf{K}$ for 
	any  $c \in [0,1)$.
	
	If $L \geq 1$, then there exists $c^* \in (0,1)$ such that for $c \in (c^*,1)$ 
	all solutions of \eqref{2} with an initial condition $\mathbf{x}_0 \in D$ converge to $\mathbf{K}$. 
\end{theorem}

If we select the zero target $\mathbf{T}=\mathbf{0}$ in  \eqref{2a}, 
the resulting control will be the proportional feedback (PF)
\begin{equation}
\label{PF_eq}
\mathbf{x}_{n+1}=\mathbf{f}((1-c) \mathbf{x}_n), \quad c \in [0,1),
\end{equation}
assuming the reduction of the state variable, which is proportional to the size of this variable \cite{gm}. 
Proportional reduction models constant effort harvesting or culling processes. 
Switching the variable transformation 
$\psi(\mathbf{x})=(1-c)\mathbf{x}$ with the map $\mathbf{f}$, we get a modified proportional feedback method (MPF)
\begin{equation}
\label{MPF} 
\mathbf{x}_{n+1}=(1-c) \mathbf{f}(\mathbf{x}_n), \quad c \in [0,1),
\end{equation} 
in which the control occurs after the process modeled by $\mathbf{f}$ takes place (e.g. reproduction). 
Applying Theorem~\ref{th1} to \eqref{MPF}, we obtain a result on the stabilization of the origin using MPF control.

\begin{corollary}
	\label{cor1}
	Assume that $D \subseteq \R^d$ is  convex, $\mathbf{0}\in D$ and that for a continuous function $\mathbf{f}\colon D \to D$,  
	the inequality $\| \mathbf{f(x)} \| \leq L \| \mathbf{x} \|$ holds for $\mathbf{x}\in D$.
	Then for $c \in (c^*,1)$ with $c^*= \min\{0,1-\frac{1}{L}\}$,  the origin is an attractor for any sequence starting 
	with $\mathbf{x}_0 \in D$ and satisfying the controlled equation \eqref{MPF}.
\end{corollary}

In Theorem \ref{th1} and Corollary \ref{cor1}, 
the constant $L$ in \eqref{3} is used to estimate the control intensity necessary to stabilize globally an equilibrium: 
the smaller $L$ is, the sooner is the global stability attained. 
Therefore, it is interesting to have easy ways to calculate $L$ for a given map. 
If $\mathbf{f}$ is globally Lipschitz continuous with $\mathbf{K}\in D$ being an equilibrium of $\mathbf{f}$, 
we can take $L$ as the global Lipschitz constant of $\mathbf{f}$, though this is not necessarily a minimal $L$. 
It is also possible to estimate $L$ if $\mathbf{f}$ is a locally Lipschitz continuous bounded function. 

\begin{lemma}
	\label{lemma4}
	Let $\mathbf{f} \colon D \to D$ be a locally Lipschitz continuous bounded function and $\mathbf{K} \in D$ be a fixed point of $\mathbf{f}$. 
	Then there exists $L\ge 1$ such that condition~\eqref{3} holds for any $\mathbf{x} \in D$.
\end{lemma}

So far we have justified the possibility to
stabilize only  a fixed point of the original map with VMTOC. 
Let us aim to stabilize an arbitrary interior point in $D$ with VMTOC.

\begin{theorem}
	\label{th2}
	Suppose $\mathbf{f}\colon D \to D$ is a continuous function which is either globally Lipschitz, or locally Lipschitz and bounded, 
	and $D$ is convex.
	Then for any $\mathbf{K}\in D$, $\mathbf{K} \not\in \partial D$ 
there exists a VMTOC for which 
	$\mathbf{K}$ is a globally asymptotically stable equilibrium.
\end{theorem}

%%%%%%%%%%%%%%%%%%%%%%%%%%%%%
\section{Applications}
\label{Applic}

\subsection{LPA model}
Consider the following age structured model designated to  describe the changes in the densities of the flour beetle \textit{Tribolium castaneum} life stages
\begin{equation}
\label{LPAmodel}
\left \{ \begin{array}{l}
L_{n+1} =b A_{n} \exp(-c_{el} L_{n} -c_{ea} A_n),\\
P_{n+1}=(1-\mu_{L}) L_{n},\\
A_{n+1}= P_n \exp(-c_{pa} A_{n}) + A_{n}(1-\mu_A),
\end{array}
\right .
\end{equation}
where $L_n$ is the number of feeding larvae, $P_n$ is the number of nonfeeding larvae, pupae and callow adults, 
and $A_n$ is the number of sexually mature adults at stage $n$, whereas  $b>0$ is the number of larval recruits per adult per unit of time in the absence of cannibalism, $c_{el}, c_{ea},c_{pa}\in \mathbb R_+$ account for the cannibalism of eggs by both larvae and adults and the cannibalism of pupae by adults,  and 
$\mu_{L}, \mu_A \in (0,1)$ are the larval and adult rates of mortality. 
We refer to \cite{desharnais2001chaos} for a more detailed explanation of the model. 

System \eqref{LPAmodel} is known as the LPA (larvae-pupae-adults) model. 
Depending on the values of the parameters, the dynamics predicted by the LPA model can be different (e.g. stable 
equilibria, periodic cycles, chaotic oscillations) and, very interestingly, 
these different types of dynamics have been demonstrated with flour beetle populations in the laboratory.

Here, we set  $b=10.45$, $c_{el}=0.01731$,  $c_{ea}=0.01310$, $c_{pa}=0.35$, $\mu_{L}=0.200$, and $\mu_A=0.96$. 
For these parameter values the LPA model exhibits chaotic oscillations and there exists a strange attractor \cite{desharnais2001chaos}. 
Figure \ref{fig:3D} shows this chaotic attractor and the oscillations for all age stages.

\begin{figure}[htp]
	\centering
	\includegraphics[width=0.45\linewidth]{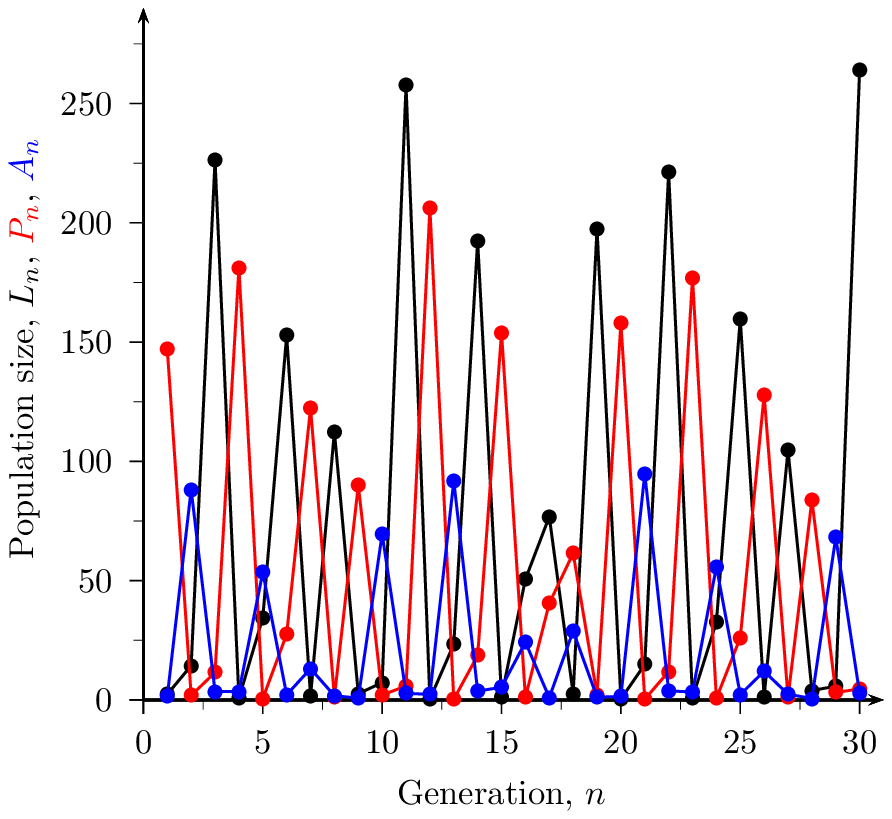}
	\quad 
	\includegraphics[width=0.5\linewidth]{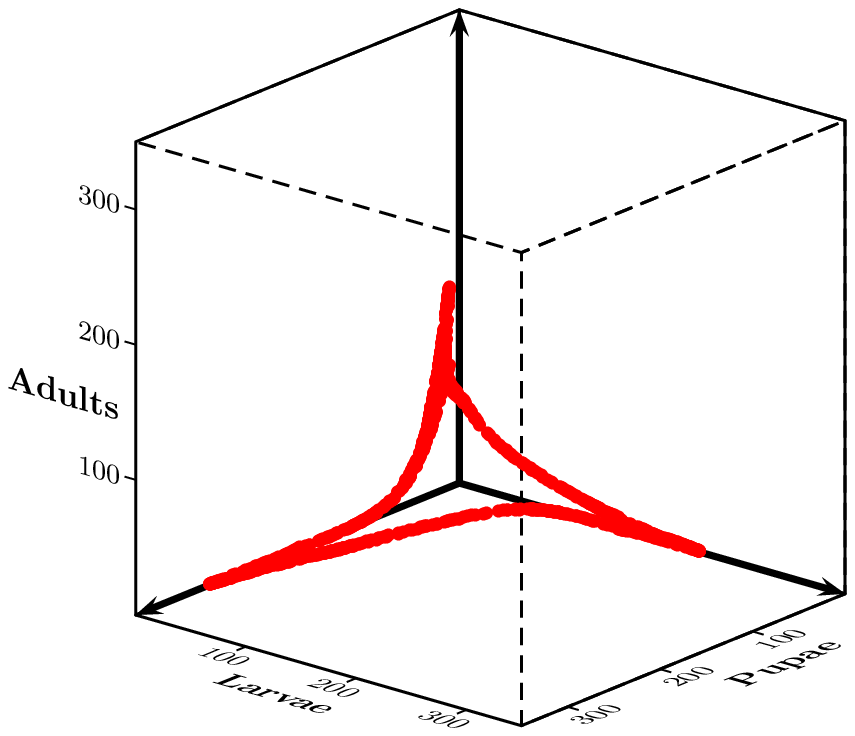} 
	\caption{
	Illustration of the chaotic dynamics of the LPA model  \eqref{LPAmodel} with the parameters given in the main text. 
	\textbf{Left:} chaotic oscillations for all age stages, we plotted 30 consecutive generations of the three age stages $L_n$ (black), 
	$P_n$ (red) and $A_n$ (blue) after discarding the first 3000 values. 
	\textbf{Right:} asymptotic strange attractor, we plotted 1000 consecutive 3-D points  $(L_n, P_n,A_n)$  after discarding the first 3000 values. Initial conditions were chosen pseudo-randomly.
 }
	\label{fig:3D}
\end{figure}

In order to stabilize this chaotic LPA model, Desharnais {\em et al} proposed the in-box control method \cite{desharnais2001chaos}. 
This method has two steps: 
first, detecting the region of the attractor more sensible to small perturbations by computing the \emph{local} Lyapunov exponents 
\cite{bailey1996local}, and, second, modifying the population (by adding a fixed number of individuals) when the population is inside that region. 
The in-box method was able to stabilize chaotic populations of the flour beetle in the laboratory, 
and the experimental data obtained was predicted correctly by numerical simulations. 
The in-box method is an empirically tested 
control of chaos strategy in age-structured population dynamics, see also \cite{fryxell2005evaluation,sah2013stabilizing,tung2016simultaneous}.   
However, the in-box method is not easy to apply. 
A controller needs to detect the region of the attractor more sensible to small perturbations and then to 
establish an action (remove or add certain type of individuals) that sends the population out of that region. 
Here, we show that, at least theoretically, VMTOC can be used to stabilize age-structured populations in a simpler way.

LPA model \eqref{LPAmodel} can be written as $\mathbf{x}_{n+1}=\mathbf{f}(\mathbf{x}_n)$ with  $\mathbf{x}=(x_1,x_2,x_3)$  
\[
\mathbf{f}(\mathbf{x})=(b x_3 \exp(-c_{el} x_1 -c_{ea} x_3), (1-\mu_{L}) x_1, x_2 \exp(-c_{pa} x_3) + x_3(1-\mu_A)).
\] 
Next, note that $f_1(\mathbf{x})\le\mbox{e}^{-1} b/ c_{ea}\approx 295,44$,  $f_2(\mathbf x)\le 0.8 
x_1$, and $f_3(\mathbf{x})\le \max\{x_2,x_3\}$ for each $\mathbf{x} \in \R^3_+$. 
Fixing the $\max$-norm $\| \mathbf{x} \| =\max_{1\leq j\leq n} |x_j|$, 
we have that  $\mathbf{\|f(x)\| \le \| x\|}$ for any $\mathbf{x}$ with $\|\mathbf{x}\|\ge \mbox{e}^{-1} b/ c_{ea}$. 
Lemma \ref{lemma_exist} guarantees that applying VMTOC to the LPA model, 
independently of the target $\mathbf{T} \in \mathbb R^3_+\setminus\{\mathbf{0}\}$ and the control intensity $c\in (0,1)$, 
there exists at least one equilibrium in the open domain 
\[
\left \{
\mathbf{x}\in \mathbb R^d_+ :  0<\|\mathbf x\|<\max\{\mbox{e}^{-1} b/ c_{ea},\|\mathbf{T}\|\} 
\right \}.
\]

Numerically, one finds that $\mathbf{K}\approx (28.0120,22.4096,4.6251)$ is a fixed point of the LPA model. Applying VMTOC with target  
$\mathbf{T}=(28.0120,22.4096,4.6251)$ to the LPA model \ref{LPAmodel} has the effect illustrated in Figure \ref{fig:LPAbifK}. 

\begin{figure}[h]
	\centering
	\includegraphics[width=0.8\linewidth]{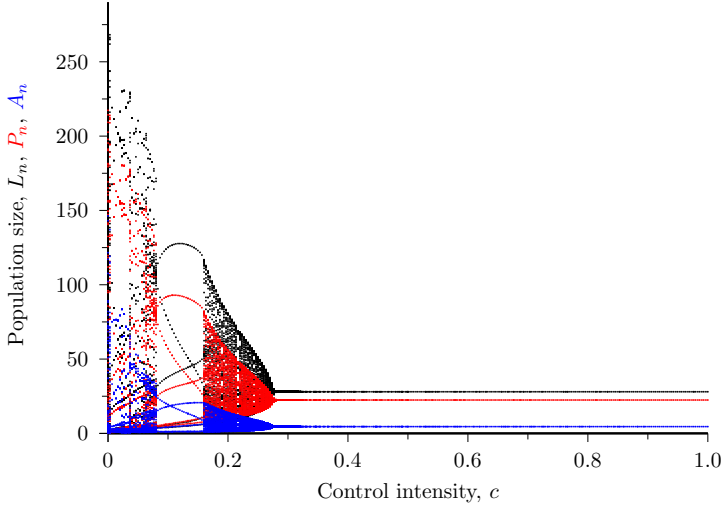}
	\caption{Stabilizing effect of VMTOC on the uncontrolled LPA model \eqref{LPAmodel}  with the parameters given in the main text. 
Target was chosen as $\mathbf{T}=(28.0120,22.4096,4.6251)$, which is an approximation of an equilibrium of the uncontrolled system. For each $c\in k/300$, $k=1,\dots,300$, we plotted 50 consecutive values of the three age stages $L_n$ (black), $P_n$ (red) and $A_n$ (blue) 
after discarding the first 3000 values. Initial conditions were chosen pseudo-randomly.}
	\label{fig:LPAbifK}
\end{figure}

This agrees with Theorem~\ref{p_esta} and Remark \ref{remark_loc_stability}. Indeed, the Jacobian of $\mathbf{f}$ is
\[
J\mathbf{f}(\mathbf{x})=
\begin{pmatrix}
-b\,c_{el}\,x_3\,{\mbox e}^{-c_{ea}\,x_3-c_{el}\,x_1} & 0 & b\,(1-c_{ea}\,x_3){\mbox e}^{-c_{ea}\,x_3-c_{el}\,x_1} \cr 
1-\mu_L & 0 & 0\cr 
0 & {\mbox e}^{-c_{pa}\,x_3} & 1-\mu_A-c_{pa}\,x_2\,{\mbox e}^{-c_{pa}\,x_3}
\end{pmatrix},
\]
and calculating numerically the value of the spectral radius of $J\mathbf{f}(\mathbf{T})$, 
we obtain $\rho(J\mathbf{f}(\mathbf{T}))\approx  1.3803$.    
Therefore, using Theorem~\ref{p_esta} and Remark \ref{remark_loc_stability}, we can guarantee that $\mathbf{K}$ is asymptotically stable for MVTOC 
if $c$ is greater than $c^*=1-\frac{1}{1.3803}\approx 0.2756$. 

For the scalar case, in \cite{TPC} 
the effect of changing the size of the target on TOC and MTOC  was studied.  
There, the focus was mainly on the size of the equilibrium, showing, for example, that this size increases as the size of the target 
increases. 
Here, we present an effect of changing the target which has not been previously reported.  
To this end, we select three different targets $\mathbf{T}_1=(30,30,200)$, $\mathbf{T}_2=(30,200,30)$, and $\mathbf{T}_3=(200,30,30)$ and use the same Euclidean norm. 
Note that choosing one of them means that the controller designs one of the age stages to be prevalent in comparison to the other two.  

Our results show that independently of the target, VMTOC will stabilize an equilibrium if the control intensity is large enough. 
However, the responses are different, not only from the point of view of the sizes of the different age stages at the equilibrium (which 
one should expect), but also the effect of increasing $c$ on the stability of the equilibrium. 
For targets $\mathbf{T}_2$ and $\mathbf{T}_3$ (see Figure \ref{fig:LPAbif2003030}), increasing $c$ 
has no negative effect on the stability: if for $c=\hat{c}$ the controlled system has a stable equilibrium, 
then it has a stable equilibrium for any $c \in (\hat{c},1)$.  
For the other target, $\mathbf{T}_1$, this is not true: 
for a certain $c=\hat{c}$, the controlled system has a stable equilibrium for $c\in (\hat{c},c_1)$, while
for $c\in (c_1, c_2)$, the adult population is still stable but both pupae and larvae have a stable two-cycle,
which has the form of bubbling, see Figure~\ref{fig:LPAbif3030200}. Increasing further the control
intensity further removes these oscillations.  
Figure~\ref{fig:LPAbif3030200} illustrates it, for $c=0.1$ the three age stages tend to an equilibrium value, 
but increasing the control intensity 
to any $c$ in approximately $(0.17,0.36)$ causes oscillations in larvae and pupae populations; if $c>0.36$, 
the three age stages tend again to their equilibrium values.

\begin{figure}[htp]
	\centering
	\includegraphics[width=0.7\linewidth]{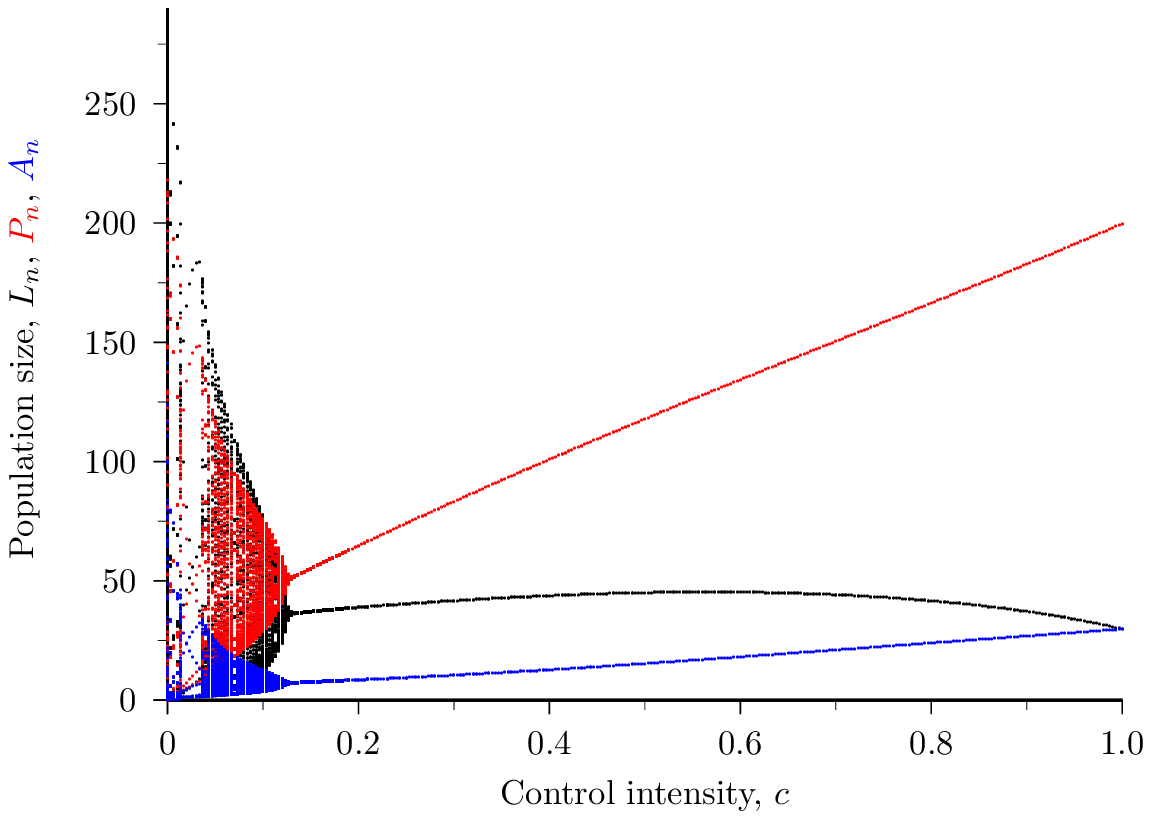}
	
	\includegraphics[width=0.7\linewidth]{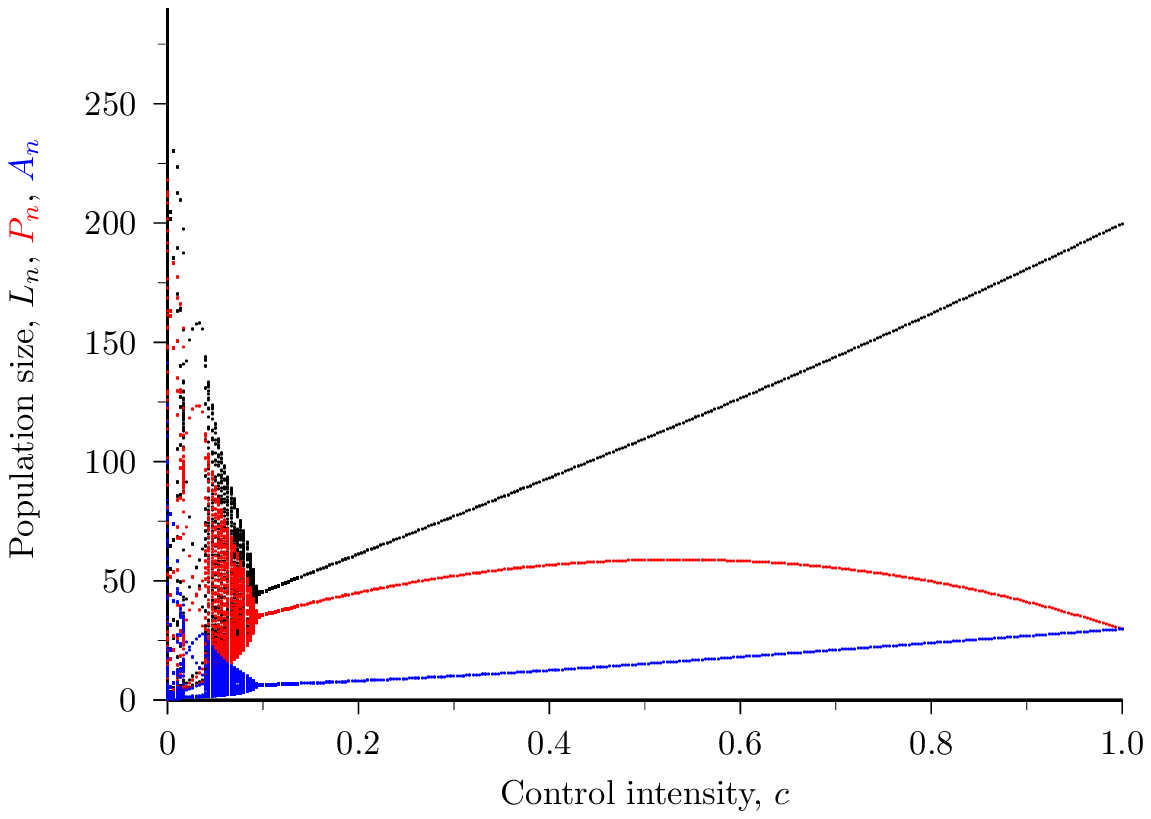}
	\caption{VMTOC applied to \eqref{LPAmodel}  with the parameters given in the main text.  
         For each $c\in k/300$, $k=1,\dots,300$, we plotted 50 consecutive values of the three age-stages $L_n$ (black), $P_n$ (red) and $A_n$ (blue) after discarding 
the first 3000 values. Initial conditions were chosen pseudo-randomly. \textbf{Top:} the target is $\mathbf{T_2}=(30,200,30)$. \textbf{Bottom:} 
the target is $\mathbf{T_3}=(200,30,30)$. }
	\label{fig:LPAbif2003030}
\end{figure}

\begin{figure}[htp]
	\centering
	\includegraphics[width=0.8\linewidth]{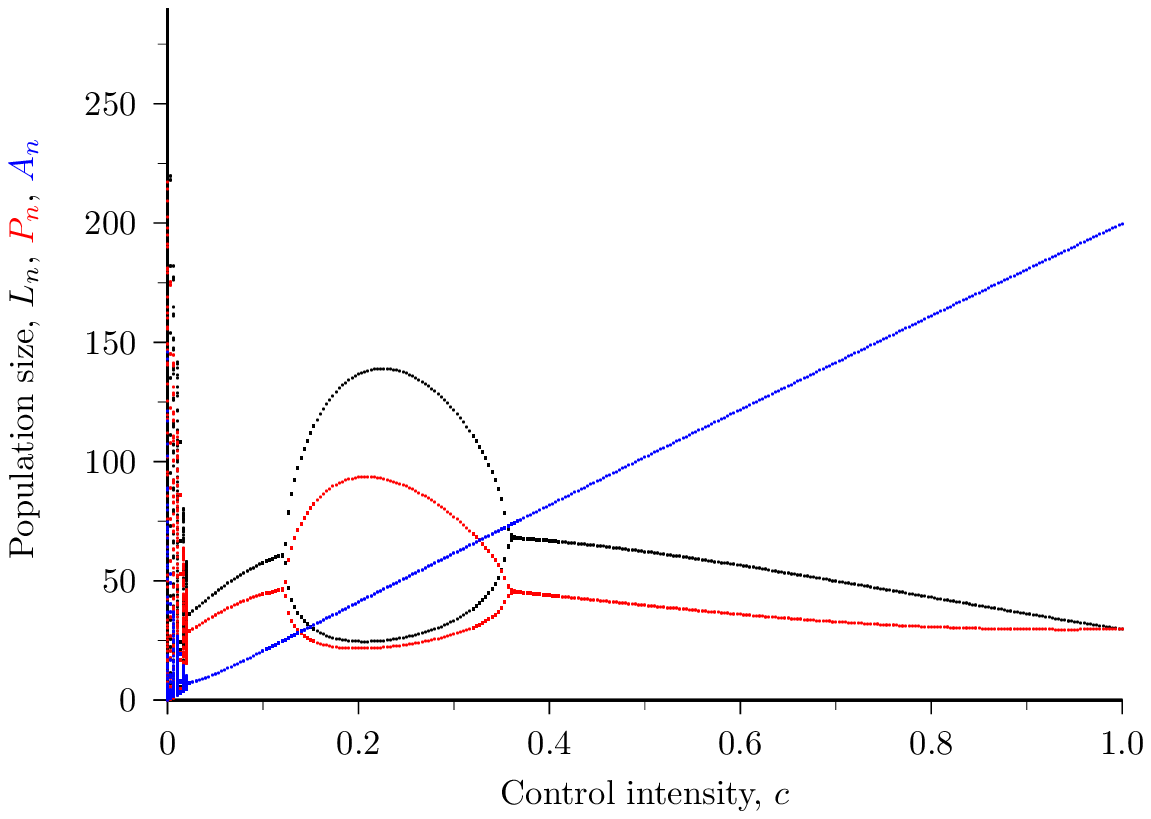}
	\caption{Bubbles in two age-classes show that increasing 
		$c$ has not always a stabilizing effect for VMTOC applied to \eqref{LPAmodel}  with the parameters given in the main text. 
    The target is $\mathbf{T_1}=(30,30,200)$. For each $c\in k/300$, $k=1,\dots,300$, we plotted 50 consecutive values of the three age-stages $L_n$, 
$P_n$ and $A_n$ after discarding the first 3000 values. Initial conditions were chosen pseudo-randomly.}
	\label{fig:LPAbif3030200}
\end{figure}

	\subsection{Higher order equations revisited}
	\label{subsec:higher}

In \cite{braverman2015stabilization}, 
we showed that using a fixed target is possible to stabilize an equilibrium of a chaotic higher order difference equation. 
Higher order difference equations arise, for instance, when studying multi-seasonal interactions in a population with  non-overlapping generations \cite{levin1976note}.
	
	Although increasing the control intensity in the method discussed in \cite{braverman2015stabilization} stabilizes an equilibrium, 
it does not follow the characteristic route from chaos of folding-period bifurcations that many stabilizing strategies show. Indeed, TOC itself presents this route for one-dimensional models---see the bifurcation diagram 
        of Figure 1 in \cite{Dattani} and compare with the bifurcation diagrams in Figures 1, 2 and 3 in \cite{braverman2015stabilization} for 
        higher-dimension models. 
Those bifurcation diagrams indicate that the stabilization of an equilibrium in higher order equations happens through a Neimark-Sacker bifurcation while in the one-dimensional case happens through a period-folding bifurcation. 
Thus, if the aim is stabilization of periodic orbits of higher order equations, direct application of TOC presented in 
\cite{braverman2015stabilization}  is not useful.

Let us illustrate with a numerical example that applying a periodic target, i.e. using VMTOC with a target $\mathbf T$ with non-equal 
components, can stabilize a periodic orbit for higher-order equations. 
We consider the delay  Ricker equation in the form of \cite{peran2015global}
        \begin{equation}
	\label{Rickerdelayed}
	u_{n+1}=u_n \exp(r-u_{n-d+1}),
	\end{equation}
	where $d\ge 2$ is a fixed natural number determining the time lag in the intraspecific regulatory mechanisms of the population. Equation \eqref{Rickerdelayed} can be rewritten as the system  
	\begin{equation}
	\label{RickerdelayedSystem}
	\mathbf{x}_{n+1}=\mathbf{f}(\mathbf{x}_n)
	\end{equation}
	with $\mathbf{x}=(x_1,x_2,\dots,x_k)$ and $\mathbf{f}(\mathbf{x})= (x_1 \exp(r-x_{k}),x_1 ,\dots,x_{k-1})$.

Let us fix $r=2$, $d=2$ and choose the target $\textbf{T}=(1,3)$. 
Figure \ref{fig:PlotUncontrolControl}B shows the effect of applying VMTOC to system \eqref{RickerdelayedSystem}. 
There, the population evolves without any control during the first 20 generations, then VMTOC with control intensity $c=0.4$ is applied. 
We can observe how a period-two orbit is stabilized. Figure \ref{fig:PlotUncontrolControl}A corresponds to 
the population dynamics without control, 
showing that the uncontrolled population has phases of low population density followed by phases with higher-density.

	\begin{figure}[htp]
				\centering
				\includegraphics[width=0.9\linewidth]{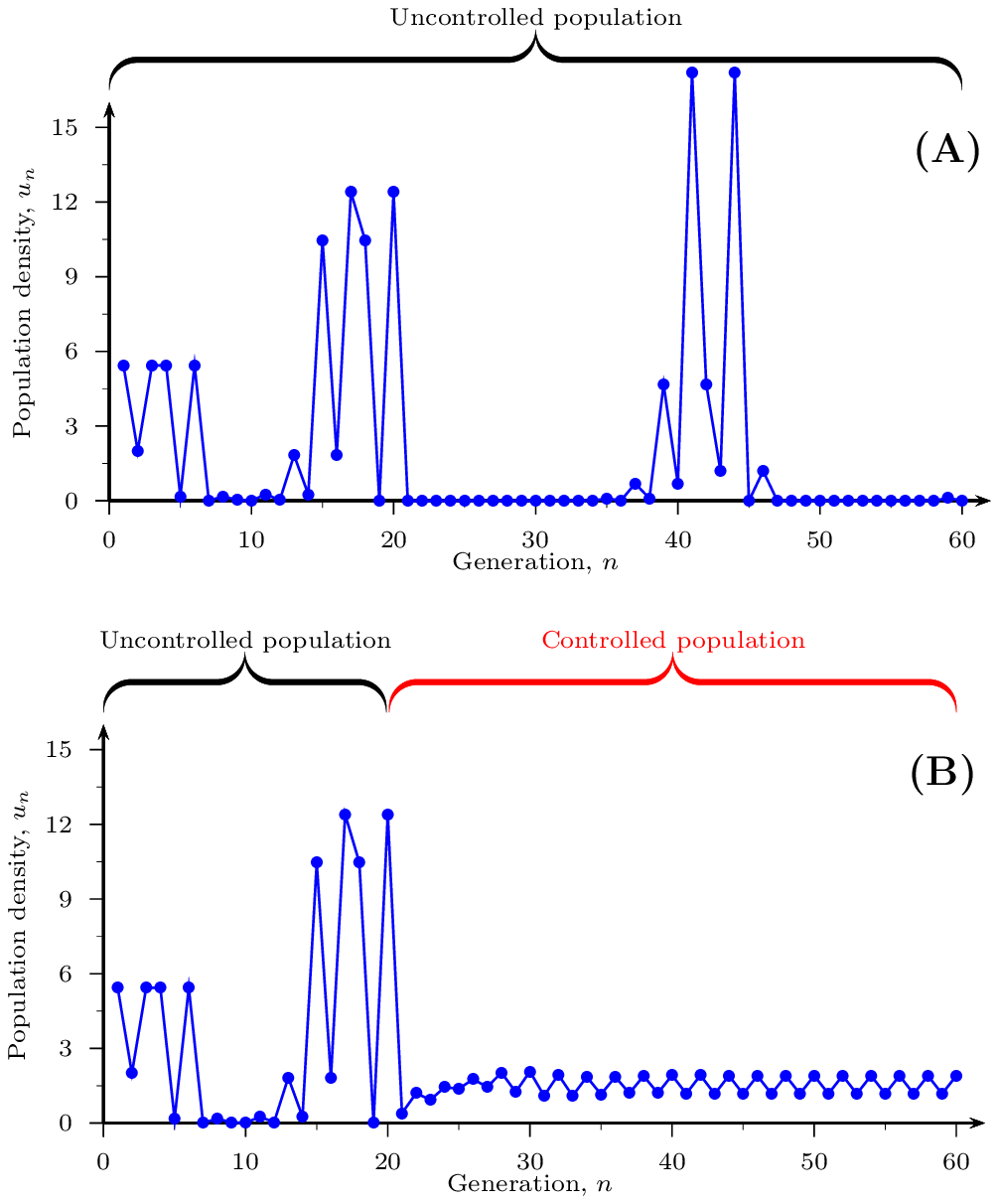}
				\caption{
			Delay Ricker equation \eqref{Rickerdelayed} uncontrolled and controlled using a periodic target $\textbf{T}=(1,3)$. 
(A) Uncontrolled population showing phases of low population density followed by phases with higher-density. 
(B) Population controlled after 20 generations.
During the first 20 generations no control is applied. 
Starting with the 20th, VMTOC is applied to  \eqref{RickerdelayedSystem} with $\textbf{T}=(1,3)$.}
				\label{fig:PlotUncontrolControl}
			\end{figure}

\section{Discussion}

Compared to other (especially one-parameter) methods,
target oriented control (TOC) has the advantage that it allows independent choice of both the stabilized state and
the control intensity. For instance, TOC can stabilize a point which 
is not an equilibrium of the original map. 
The vector modifications VTOC and VMTOC inherit these two properties of scalar TOC: a variety of originally unstable states 
can be stabilized. Certainly the minimal control sufficient for stabilization would depend on the choice of the target $\mathbf{T}$
and the state to be stabilized. However, as we can observe, the ability to stabilize an arbitrary state can be quite cost-involved.
The cost efficiency of TOC is discussed in \cite{Dattani} and analyzed numerically. 
If we describe the eventual cost-per-step $P$ as the ``average" control perturbation 
\begin{eqnarray}
P & = & \limsup_{n,k \to\infty} \frac{1}{k} \sum_{j=n}^{n+k-1} \| c\mathbf{T}+(1-c)\mathbf{f}(\mathbf{x}_n) - \mathbf{f}(\mathbf{x}_n) \|
\nonumber \\
& = & c \limsup_{n,k \to\infty} \frac{1}{k} \sum_{j=n}^{n+k-1} \| \mathbf{f}(\mathbf{x}_n) - \mathbf{T} \|
\label{cost}
\end{eqnarray}
then, once a state $\mathbf{x}^{\ast}$ is stabilized, the eventual cost-per-step becomes
$$
P=  c \| \mathbf{f}(\mathbf{x}^{\ast}) - \mathbf{T} \|.
$$
If the state to be stabilized is a fixed point $\mathbf{K}=\mathbf{f}(\mathbf{K})$ 
of the original vector map, the cost-per-step may be quite high 
in the transient period but eventually tends to zero.
From the above, the closer we can choose the target to the image of the stabilized state, more efficiency can be achieved.
However, the minimal stabilizing $c$ also depends on the choice of $\mathbf{x}^{\ast}$ and $\mathbf{T}$. 
The cost optimization of VMTOC is a separate question, and its solution is not in the framework of the current paper.
Also, simplified estimate \eqref{cost} does not take into account the following factors:
\begin{itemize}
\item
the costs of culling and restocking can be different;
\item
the costs may strongly depend on either age stage or patch location. 
\end{itemize}

Moreover, culling and restocking of certain age-stages or locations can be problematic.
This brings us to the discussion of some further VMTOC generalization.
Originally, TOC included two parameters: the target $T$ and the control intensity $c$.
We considered VMTOC in $\R^d$ with $d+1$ parameters involved. 
This allowed us to obtain results on the possibility of stabilization for a variety of states,
and to estimate the minimal control intensity $c$ in each case. 
However, it is quite natural also to consider a $2d$-parameter method with a diagonal control matrix $C$, where 
the controls $c_{j}=c_{jj}$ are applied to each $j$-th stage. 
With this modification, VMTOC has  the form
\begin{equation}
\label{2_modified}
\mathbf{x}_{n+1}=C\mathbf{T}+(I_d-C)\mathbf{f}(\mathbf{x}_n), \quad \mathbf{T} \in \R^d, \; C=diag\{c_{j} \}, \, c_j\in [0,1),
\end{equation}
where $I_d$ is the $d\times d$ identity matrix and $j=1,2,\dots,d$.
If some stages cannot be controlled, the corresponding $c_j$ will be zero.

A higher order difference equation is a particular case of the one-dimensional vector equation, and VMTOC allows to stabilize 
periodic orbits of higher order equations, see the example of stabilization for a two-orbit of the delayed Ricker model above.
Let us note that the same technique applies to stabilizing a periodic $k$-orbit of a vector map.
Indeed, denote by $\mathbf{f}^k$ 
the $k$-th iterate of $\mathbf{f}$.
Then stabilizing a periodic $k$-orbit of $\mathbf{f}$ is equivalent to stabilizing a certain state of $\mathbf{f}^k$.
All the results of the present paper apply to this case. 
We can also consider a control type similar to \eqref{2_modified}, with either $c_j=c$ or zeros on the main diagonal of 
$C$. In the case of $c_1= \dots=c_d=c$, all the other $c_j$ for stabilization of $\mathbf{f}^k$ being zeros, the result is the pulse 
control \cite{Chaos2014}. If all $c_j=0$, except $c_k, c_{k+d}, \dots$, this corresponds to a control of only one age stage or patch. 
While it may definitely be problematic to achieve stabilization goals with this limited type of control, it is still an interesting 
question whether controlling one stage only (for example, juveniles), we can reduce the risk of extinction and population fluctuations.

% % % % % % % % % % % %
% % % % % % % % % % % %

\appendix

\section{Appendix}

The next result shows that VTOC and MVTOC are topologically conjugate, thus they have the same dynamics. 
We recall that two maps $\phi$ and $\psi$ are topologically conjugate if  
there is a homeomorphism $h$ such that $\phi \circ h= h  \circ \psi$.

\begin{lemma} 
\label{lemma3} 
Assume $\mathbf f \colon D \to D$, with $D$ convex and $\mathbf{T}\in D$. 
Then the difference equations \eqref{2a} and \eqref{2} are topologically conjugate.
\end{lemma}  
\begin{proof} 
We are going to show that the maps defining the difference equations \eqref{2a} and \eqref{2} are topologically conjugate. We begin by 
defining such maps.    
	
	Consider the map $\varphi(\mathbf{x})=c\mathbf{T}+(1-c)\mathbf{x}$  from $D$ to $\varphi(D)$. Moreover, since $D$ is convex $ \varphi(D) \subset D$ and the map $\mathbf{f}\circ \varphi\colon D \to D$ is well defined. Clearly,  map $\mathbf{f}\circ \varphi$ defines the recurrent relation in equation \eqref{2a}. On the other hand, note that after the first iterate the solutions of \eqref{2} belong to $\varphi(D)$. Therefore, we have that after the first iterate the map that defines the recurrence given by  \eqref{2} is  $\varphi \circ \mathbf{f} \colon \varphi(D) \to \varphi(D)$.  
	
	It is easy to check that  $\varphi$ is a homeomorphism from $D$ to $ \varphi(D)$ and obviously $(\varphi \circ \mathbf{f})\circ\varphi=\varphi \circ ( \mathbf{f}\circ \varphi)$.  Hence, $\mathbf{f}\circ \varphi$ and $\varphi \circ \mathbf{f}$ are topologically conjugate.	    
\end{proof}

%new

\begin{proof} {\em of Lemma 1}
An equilibrium $\mathbf{x}^{\ast}$ of VMTOC is a fixed point of the map
\begin{equation}
\label{eq_fp}
\mathbf{g(x)}= c \mathbf{T} + (1-c) \mathbf{f(x)}.
\end{equation}
Since $\mathbf{T}\in \mathbb{R}^d_{+}\setminus \{\mathbf{0} \}$ and $\mathbf{f}\colon \mathbb{ R}^d_+\to \mathbb{ R}^d_+$,
we have for every $c \in (0, 1)$ that  $c \mathbf{T} + (1-c) \mathbf{f(0)}\in \mathbb{R}^d_{+}\setminus \{\mathbf{0} \}$ and
\[\|\mathbf{g(0)}\| = \| c \mathbf{T} + (1-c) \mathbf{f(0)}\| > \|\mathbf{0}\|.
\]
Hence, by the continuity of $\mathbf{g}$ and the norm, 
for each fixed $c \in (0, 1)$ is possible to find $0<m<M$ 
such that 
\[
\|\mathbf{g(x)} \|  \ge  \|\mathbf{x}\|, \quad \mathbf{x}\in \mathbb R^d_{+} , \|\mathbf{x}\|=m.
\]
On the other hand, we have for any $\mathbf{x}\in \mathbb R^d_{+}$ with $\|\mathbf{x}\|\ge \max\{M,\|\mathbf{T}\|\}$, 
\begin{align*}
\|\mathbf{g}(\mathbf{x})\| & = \|c \mathbf{T}   +   (1-c) \mathbf{f}(\mathbf{x})\| \le \|c \mathbf{T} \|  
+ \|  (1-c) \mathbf{f}(\mathbf{x})\| 
\\& 
= c\| \mathbf{T} \|  + (1-c)\|   \mathbf{f}(\mathbf{x})\|
\le c\| \mathbf{x} \|  + (1-c)\| \mathbf{x}\|= \| \mathbf{x}\|.
\end{align*}

Thus, by Krassnosel'ski{\u \i} Fixed-point Theorem for cone-compressing operators (see e.g. \cite[Theorem 7.12]{granas2013fixed}), 
the map $\mathbf{g}$ has at least one fixed point ${\mathbf x}^{\ast}$ in the set 
\[
\left \{
\mathbf{x}\in \mathbb R^d_+ :  0<\|\mathbf x\|<\max\{M,\|\mathbf{T}\|\} \right \}
\] for every $c \in (0, 1)$. 

Finally, the last statement of the lemma follows from noticing that an equilibrium $\mathbf{x}^*$ of VMTOC satisfies 
	$$
	\mathbf{x}^*=c \mathbf{T} + (1-c) \mathbf{f}(\mathbf{x}^*).
	$$ 
\qed
\end{proof}

\begin{proof} {\em of Lemma 2}
	If $\mathbf{K}=\mathbf{f}(\mathbf{K})$, we take $\mathbf{T}_\mathbf{K}=\mathbf{K}$ and any $c_\mathbf{K}\in (0,1)$. 
        Let $\mathbf{K} \neq \mathbf{f}(\mathbf{K})$.
	Consider the set of points 
	\[
	R = \{\mathbf{f}(\mathbf{K})+ \alpha (\mathbf{K}-\mathbf{f}(\mathbf{K})): \alpha>1\},
	\]
	which is the ray in the direction $\mathbf{K}-\mathbf{f}(\mathbf{K})$ starting at $\mathbf{K}$. Using that $\mathbf{K} \not\in \partial D$, we have that $R\cap D\neq \emptyset$. Therefore,   
	there exists $\alpha>1$  such that $\mathbf{T_K}:=\alpha \mathbf{K}+(1-\alpha)\mathbf{f(K)}  \in D$. 
	Fixing $c_{\mathbf{K}}=\alpha^{-1}\in (0,1)$ and 
	noticing that
	\begin{eqnarray*}
		\mathbf{g}(\mathbf{K})=c_{\mathbf K} \mathbf{T_K}+ (1-c_\mathbf{K})\mathbf{f(K)} & = & 
		\alpha^{-1} \left( \alpha \mathbf{K}+(1-\alpha)\mathbf{f(K)} \right) + (1-\alpha^{-1})\mathbf{f(K)} \\
		& = & \mathbf{K}+ \frac{1-\alpha}{\alpha} \mathbf{f(K)} + \frac{\alpha-1}{\alpha} \mathbf{f(K)} = \mathbf{K}
	\end{eqnarray*}
	concludes the proof.\qed
\end{proof}

\begin{proof} {\em of Lemma 3}
	We note that it is sufficient to consider only controls $c \in (0,1)$, since for $c=0$ we have the original map.
	Let $\phi_1(\mathbf{x})=c_1 \mathbf{T}_1+(1-c_1)\mathbf{x}$ be applied first, and next another argument transformation 
	$\phi_2(\mathbf{x})=c_2 \mathbf{T}_2+(1-c_2)\mathbf{x}$. Then
	\begin{eqnarray*}
		\phi_2\left(\phi_1(\mathbf{f}(\mathbf{x})) \right) & = & c_2\mathbf{T}_2+(1-c_2) \left[c_1\mathbf{T}_1+(1-c_1) \mathbf{f}(\mathbf{x})\right]
		\\ &=& c_2 \mathbf{T}_2+(1-c_2)c_1 \mathbf{T}_1+(1-c_1)(1-c_2)\mathbf{f}(\mathbf{x})= c\mathbf{T}+(1-c) \mathbf{f}(\mathbf{x}),
	\end{eqnarray*}
	where 
	$$ c=c_1(1-c_2)+c_2, \quad \mbox{ and } \quad \mathbf{T}=\frac{c_1(1-c_2)}{c_1(1-c_2)+c_2} \mathbf{T}_1+\frac{c_2}{c_1(1-c_2)+c_2} \mathbf{T}_2.$$
	Since  $c_1,c_2 \in (0,1)$,
	also $c=c_1(1-c_2)+c_2 \in (0,1)$. On the other hand, the vector $\mathbf{T}=\alpha \mathbf{T}_1+(1-\alpha) \mathbf{T}_2$,
	where $\alpha = c_1(1-c_2)/c\in (0,1)$, therefore by the convexity of $D$ the target $\mathbf{T} \in D$.\qed
\end{proof}

\begin{proof} {\em of Theorem 1}
Recall that, by linearization, any equilibrium $\mathbf{p}_c$ of VMTOC is asymptotically stable if the spectral radius  
of the Jacobian matrix  $J\mathbf g$ of  $\mathbf{g(x)}=c\mathbf{T}+(1-c)\mathbf{f(x)}$ 
at $\mathbf{p}_c$ is smaller than $1$.

By a well-known bound for the eigenvalues of a matrix (see e.g. \cite[Corollary 6.1.5]{horn2012matrix}), 
we have that any eigenvalue $\lambda$ of  $J\mathbf g(\mathbf{p}_c)$
satisfies
\[
|\lambda | \le  (1-c) \min \left \{\max_{j=1,\dots, d} \sum_{i=1}^{d} 
\left | \frac{\partial f_{j}}{\partial x_i}(\mathbf{p}_c) \right |, \max_{j=1,\dots,d}\sum_{i=1}^{d} \left | \frac{\partial f_{i}}{\partial x_j}(\mathbf{p}_c) \right |   \right \}.
\]
Therefore, for $c\in (c^*,1)$  the  eigenvalues of $J\mathbf g(\mathbf{p}_c)$ have modulus smaller than $1$.  \qed 
\end{proof}

\begin{proof} {\em of Theorem 2}
	If $L$ in \eqref{3} satisfies $L \in (0,1)$ and $\mathbf{x}_n$ is a solution of \eqref{2} with an initial 
	condition $\mathbf{x}_0 \in D$  and $c \in [0,1)$, then for  $n\in\mathbb{N}$
	$$
	\| \mathbf{x}_{n+1} - \mathbf{K} \| = \|c\mathbf{K} +(1-c) \mathbf{f}(\mathbf{x}_{n}) - \mathbf{K} \| =(1-c)\|  \mathbf{f}(\mathbf{x}_{n})-\mathbf{K} \| \leq L \| \mathbf{x}_{n}-\mathbf{K} \|.
	$$
	 For any $\varepsilon \in (0, \|\mathbf{x}_0 -\mathbf{K} \|)$, we have 
	$$\| \mathbf{x}_n - \mathbf{K} \| \leq L^n \| \mathbf{x}_0 -\mathbf{K} \| < \varepsilon \mbox{~~whenever~~~} n> \left. \ln \left( \frac{\varepsilon}{\| \mathbf{x}_0 -\mathbf{K} \|} \right) 
	\right/ \ln L,
	$$
	so $\mathbf{x}_n \to \mathbf{K}$ as $n \to \infty$.
	
	Next, let  $L\ge 1$. Denote
	\begin{equation}
	\label{cstar}
	c^*= \left\{ \begin{array}{ll} 0, & ~~L=1, \\ {\displaystyle 1 - \frac{1}{L}}, & ~~L>1,
	\end{array} \right.
	\end{equation}
	and assume that  $c \in (c^*,1) \subseteq (0,1)$. Then $(1-c)L<1$, denote $\theta=(1-c)L \in (0,1)$. We have 
$1-c=\theta/L$, $c\mathbf{T}=c\mathbf{K}$ and
	\begin{eqnarray*}
		\| \mathbf{x}_{n+1}-\mathbf{K} \| & = & \|(1-c)\mathbf{f}(\mathbf{x}_n)+c\mathbf{K-K}\|= 
\|(1-c)[\mathbf{f}(\mathbf{x}_n)-\mathbf{K}]\| = (1-c)\| \mathbf{f}(\mathbf{x}_n) -\mathbf{K} \|
		\nonumber \\ &=& \frac{\theta}{L} \|\mathbf{f}(\mathbf{x}_n)-\mathbf{K} \|  \leq \frac{\theta}{L} L \| \mathbf{x}_n-\mathbf{K} \| = \theta \| \mathbf{x}_n-\mathbf{K} \|.  
		\nonumber 
	\end{eqnarray*}
	By induction,
	\begin{equation*}
	\label{contraction}
	\| \mathbf{x}_{n+1}-\mathbf{K} \| \leq \theta^n  \| \mathbf{x}_0-\mathbf{K} \|.
	\end{equation*}
	Therefore, $\| \mathbf{x}_{n+1}-\mathbf{K} \| < \varepsilon < \| \mathbf{x}_0 -\mathbf{K} \|$ for 
	any $\displaystyle n > \left. \ln \left( \frac{\varepsilon}{\| \mathbf{x}_0 -\mathbf{K} \|} \right) \right/ \ln \theta$.
	Thus, $\lim\limits_{n \to \infty} \mathbf{x}_n=\mathbf{K}$. Moreover, $\|\mathbf{x}_{n}-\mathbf{K} \|$ decays at least geometrically,
	which concludes the proof. \qed
\end{proof}

\begin{proof} {\em of Lemma 4}
	Let $M>0$ be an upper bound of $\mathbf{f}$:
	\begin{equation}   
	\label{bound}
	\| \mathbf{f(x)} \| \leq M, ~~\mathbf{x} \in D. 
	\end{equation} 
	Since $\mathbf{f}$ is locally Lipschitz continuous,  for $\mathbf x$ in the intersection of the ball $\|\mathbf{ x-K} \| \leq \|\mathbf{K}\|$ with $D$,
	there is a constant $\tilde{L}>0$ such that $\| \mathbf{f(x)-K} \| \leq \tilde{L} \| \mathbf{x-K}\|$.
	Next, let $\mathbf{x}\in D$, $\| \mathbf{x-K} \| > \|\mathbf{K}\|$. Then, by \eqref{bound},  
	\begin{align*}
	\|\mathbf{f(x)-K} \|  & \leq  \| \mathbf{f(x)}\| + \| \mathbf{K} \| \leq M + \| \mathbf{K} \|   =  \left( \frac{M}{\| \mathbf{K} \|}+1 \right) \| \mathbf{K} \|  <   \left( \frac{M}{\| \mathbf{K} \|}+1 \right) \| \mathbf{x-K} \|.
	\end{align*}
	Finally, choosing
	\begin{equation} 
	\label{bound1}
	L= \max \left\{ \tilde{L}, \frac{M}{\| \mathbf{K} \|}+1 \right\},
	\end{equation}
	we obtain that inequality \eqref{3} is satisfied for any $\mathbf{x} \in D$. \qed
\end{proof}

\begin{proof} {\em of Theorem 3}
	By Lemma~\ref{lemma2} there exist $c_{\mathbf K}\in (0,1)$ and $\mathbf{T_K}\in D$ such that $\mathbf{K}$ is an equilibrium  of 
	$\mathbf{g}(\mathbf{x})=c_\mathbf{K} \mathbf{T_K}+(1-c_\mathbf{K})\mathbf{g}(\mathbf{x})$. 
	By Lemma~\ref{lemma4}, $\mathbf{f}$ satisfies \eqref{3} with some constant $L_1$ instead of $L$. Then
	\begin{eqnarray*}
		\| \mathbf{g}(\mathbf{x}) - \mathbf{K} \| & = & \| \mathbf{g}(\mathbf{x}) - \mathbf{g}(\mathbf{K}) \| = \| c_\mathbf{K} \mathbf{T_K}+(1-c_\mathbf{K})\mathbf{f(x)} - c_\mathbf{K} \mathbf{T_K}+(1-c_\mathbf{K})\mathbf{f(K)} \|  \\
		& = & (1-c_\mathbf{K}) \| \mathbf{f(x) - f(K)} \| \leq (1-c_{\mathbf{K}}) L_1 \| \mathbf{x - K} \|, 
	\end{eqnarray*}
	thus condition~\eqref{3} holds for $\mathbf{g}$ with $L=(1-c_{\mathbf K})L_1$.
	
	Now, applying Theorem~\ref{th1}, we obtain that there exists $c^* \in [0,1)$ such that for $\hat{c} \in (c^*,1)$ and $\mathbf{T=K}$,
	all solutions of \eqref{2} with $\mathbf{g}$ instead of $\mathbf{f}$, $\hat{c}$ instead of $c$ and $\mathbf{x}_0 \in D$ converge to 
	$\mathbf{K}$.
	By Lemma~\ref{lemma1}, a combination of two VMTOCs is a VMTOC. Thus, if we 
	choose $c=c_\mathbf{K}(1-\hat{c})+\hat{c}$, where $\hat{c} \in (c^*,1)$, $c_\mathbf{K}$ is defined above, and
	$$
	\mathbf{T}=\frac{c_\mathbf{K}(1-\hat{c})}{c_\mathbf{K}(1-\hat{c})+\hat{c}} \mathbf{T_K}+\frac{\hat{c}}{c_\mathbf{K}(1-\hat{c})+\hat{c}} \mathbf{K},
	$$
	we get that $\mathbf{K}$ is a global attractor of the combination of VMTOCs.\qed
\end{proof}

\begin{acknowledgements}
%E. Braverman was partially supported by the NSERC research grant RGPIN-2015-05976. 
%D.~Franco  was partially supported by the Spanish Ministerio de Econom\'{\i}a y Competitividad and FEDER, grant 
%MTM2013-43404-P. 

The authors are grateful to the anonymous referees whose valuable comments contributed to the presentation of the 
results of the paper.
\end{acknowledgements}

% BibTeX users please use one of
%\bibliographystyle{spbasic}      % basic style, author-year citations
%\bibliographystyle{spmpsci}      % mathematics and physical sciences
%\bibliographystyle{spphys}       % APS-like style for physics
%\bibliography{}   % name your BibTeX data base

%\bibliographystyle{abbrv}
%\bibliography{BiblioVMTOC}
%\end{document}

\end{document}